\documentclass[12pt]{article}

\usepackage{amsmath,amssymb,amsthm,mathtools}
\usepackage{algorithm}
\usepackage{algorithmic}
\usepackage{booktabs}
\usepackage{array}
\usepackage{float}
\usepackage{url}

\newtheorem{theorem}{Theorem}
\newtheorem{proposition}[theorem]{Proposition}
\newtheorem{corollary}[theorem]{Corollary}

\DeclareMathOperator{\Irr}{Irr}
\DeclareMathOperator{\End}{End}

\DeclareMathOperator{\Span}{span}
\DeclareMathOperator{\Tr}{Tr}

\DeclareMathOperator{\Ind}{Ind}
\DeclareMathOperator{\Res}{Res}

\title{From Characters to Matrices: An Elementary Construction of Irreducible Representations of Finite Groups}

\author{Yu Hsuan Hsieh \and Ming-Hsuan Kang}

\date{}

\begin{document}

\maketitle

\begin{abstract}
Let \(G\) be a finite group and let \(\chi\) be an ordinary irreducible
character.  We give an elementary algorithm which constructs explicit
matrices affording \(\chi\).  The regular representation provides a canonical
ambient representation, and the usual central idempotent projects onto the
\(\chi\)-isotypic component.  The main step is then to maximize the squared
norm of a diagonal matrix coefficient on the unit sphere of this component.
The maximum is \(1/\chi(1)\), and it is attained precisely by vectors whose
cyclic span is an irreducible subrepresentation affording \(\chi\).  Thus the
construction reduces the passage from characters to matrices to a concrete
optimization problem.

The same extraction method applies inside any smaller ambient representation
containing \(\chi\), such as an induced representation from a subgroup.  We complement the theory with a discussion on dimension reduction, robust numerical implementation, and an explicit
\(S_4\) example.
\end{abstract}


\medskip


\section{Introduction}

A character table records traces, but not matrices.  Thus, even when an
irreducible character \(\chi\) of a finite group \(G\) is known, constructing
explicit matrices affording \(\chi\) remains a separate problem.

Over \(\mathbb C\), irreducible characters determine irreducible
representations up to equivalence, so such matrices exist in principle.
However, the standard argument is essentially existential.  Existing
computational approaches, such as the GAP package \texttt{REPSN}, use
structural information, subgroup methods, and linear algebra; see
\cite{Dabbaghian2005,DabbaghianDixon2010,RepsnManual}.

The purpose of this note is to give an elementary construction of explicit
matrix representations from irreducible characters.  The input consists only
of the group law of a finite group \(G\) and an irreducible character \(\chi\).
The regular representation \(\mathbb C[G]\) gives a canonical ambient
representation containing every irreducible representation.  Applying the
central idempotent attached to \(\chi\) gives the \(\chi\)-isotypic component.
The remaining problem is to select one irreducible summand inside this
component.

Our selection criterion is variational.  For a unit vector \(v\) in the
\(\chi\)-isotypic component, consider the diagonal matrix coefficient
\[
        f_v(g)=\langle v,\rho(g)v\rangle .
\]
We show that the squared norm of \(f_v\) is at most \(1/\chi(1)\), and that
equality holds precisely when the cyclic span of \(v\) is an irreducible
subrepresentation affording \(\chi\).  Thus one way to pass from characters to
explicit matrices is to solve a concrete maximization problem on a sphere.

The construction has a universal version and a practical version.  The
universal version uses the regular representation, so it applies to every
irreducible character of every finite group.  The practical version uses any
smaller ambient representation containing \(\chi\).  In particular, if
\(\theta\) is a character of a subgroup \(H\leq G\) and
\[
        \langle \chi,\Ind_H^G\theta\rangle_G
        =
        \langle \Res_H^G\chi,\theta\rangle_H
        >
        0,
\]
then the same extraction method may be applied inside \(\Ind_H^G\theta\),
which can be much smaller than the regular representation.

Unlike the classical matrix-coefficient projectors used to decompose an
isotypic component once a matrix realization is already known
\cite[Section 2.7]{Serre}, the present method starts only from the character
\(\chi\) and first constructs one irreducible copy.

A Python implementation of the algorithm, together with the \(S_4\) example
from Section~\ref{sec:s4-example}, is included in the ancillary files of the
arXiv version of this paper.

During the preparation of this manuscript, the authors used ChatGPT as a tool
for editing, code prototyping, and organizing computational examples.  The
authors take full responsibility for the mathematical content, proofs, and
implementation.

\section{Preliminaries}

Throughout the paper, \(G\) is a finite group and all representations are complex representations.
Hermitian inner products are taken to be conjugate-linear in the first variable and linear in the second.
A representation \(\rho:G\to GL(V)\) may always be made unitary by averaging an inner product over \(G\).
Thus, in the sequel, we assume that all representations are unitary.

For functions \(a,b:K\to\mathbb C\) on a finite group \(K\), we use the
standard character pairing, linear in the first variable:
\[
        \langle a,b\rangle_K
        =
        \frac1{|K|}
        \sum_{k\in K}
        a(k)\overline{b(k)}.
\]

Let \(\chi\in \Irr(G)\) be an irreducible character of degree
\[
        d=\chi(1).
\]
For a representation \(\rho:G\to GL(V)\), the \(\chi\)-isotypic projector is
\[
        P_\chi
        =
        \frac{d}{|G|}
        \sum_{g\in G}
        \overline{\chi(g)}\,\rho(g).
\]
This is the standard central idempotent acting on \(V\).
Its image is the \(\chi\)-isotypic component
\[
        V_\chi=P_\chi V.
\]

If \(v\in V\), define the diagonal matrix coefficient
\[
        f_v(g)
        =
        \langle v,\rho(g)v\rangle.
\]
We write
\[
        F(v)
        =
        \frac1{|G|}
        \sum_{g\in G}
        |\langle v,\rho(g)v\rangle|^2.
\]

\section{A Variational Characterization of Irreducible Summands}

The following theorem is the central observation.

\begin{theorem}
\label{thm:main}
Let \(\rho:G\to U(V)\) be a unitary representation, and let \(V_\chi\) be the \(\chi\)-isotypic component, where \(\chi\in\Irr(G)\) has degree \(d\).
If \(v\in V_\chi\) and \(\|v\|=1\), then
\[
        F(v)
        =
        \frac1{|G|}
        \sum_{g\in G}
        |\langle v,\rho(g)v\rangle|^2
        \leq
        \frac1d.
\]
Equality holds if and only if the cyclic subspace
\[
        \Span\{\rho(g)v:g\in G\}
\]
is an irreducible subrepresentation affording \(\chi\).
\end{theorem}

\begin{proof}
Since \(V_\chi\) is isotypic, there is a unitary \(G\)-module isomorphism
\[
        V_\chi \cong W\otimes M,
\]
where \(W\) is an irreducible representation with character \(\chi\), \(\dim W=d\), and \(M\) is a multiplicity space.
The group acts by
\[
        \rho(g)=\pi(g)\otimes I_M,
\]
where \(\pi:G\to U(W)\) affords \(\chi\).

Choose an orthonormal basis \(m_1,\dots,m_r\) of \(M\).
Write
\[
        v=\sum_{j=1}^r w_j\otimes m_j,
        \qquad
        w_j\in W.
\]
Define the positive operator
\[
        T_v
        =
        \sum_{j=1}^r w_j w_j^*
        \in \End(W),
\]
where \(w_jw_j^*\) denotes the rank-one operator \(x\mapsto w_j\langle w_j,x\rangle\).
This convention agrees with our choice that the inner product is
conjugate-linear in the first variable.
Then
\[
        \Tr(T_v)
        =
        \sum_j\|w_j\|^2
        =
        \|v\|^2
        =
        1.
\]
Moreover,
\[
        f_v(g)
        =
        \sum_j
        \langle w_j,\pi(g)w_j\rangle
        =
        \Tr(T_v\pi(g)).
\]

By Schur orthogonality in operator form,
\[
        \frac1{|G|}
        \sum_{g\in G}
        |\Tr(A\pi(g))|^2
        =
        \frac1d
        \Tr(AA^*)
        \qquad
        \text{for all } A\in \End(W).
\]
This is just the usual Schur orthogonality relation for matrix coefficients,
applied after expanding \(A\) in an orthonormal basis of matrix units.
Taking \(A=T_v\), and using \(T_v=T_v^*\), gives
\[
        F(v)
        =
        \frac1d
        \Tr(T_v^2).
\]
Since \(T_v\) is positive semidefinite and \(\Tr(T_v)=1\),
\[
        \Tr(T_v^2)\leq (\Tr T_v)^2=1.
\]
Hence
\[
        F(v)\leq \frac1d.
\]

Equality holds if and only if \(T_v\) has rank one.
Equivalently, the vectors \(w_1,\dots,w_r\) are all proportional.
In that case,
\[
        v=w\otimes m
\]
for some nonzero \(w\in W\) and some unit vector \(m\in M\).
Then
\[
        \Span\{\rho(g)v:g\in G\}
        =
        \Span\{\pi(g)w:g\in G\}\otimes \mathbb C m
        =
        W\otimes \mathbb C m,
\]
which is an irreducible subrepresentation affording \(\chi\).

Conversely, if the cyclic subspace generated by \(v\) is irreducible, then in the decomposition \(W\otimes M\) it must be of the form \(W\otimes L\) with \(\dim L=1\).
Thus \(v\) is a pure tensor \(w\otimes m\), so \(T_v\) has rank one and equality holds.
\end{proof}

\begin{corollary}
\label{cor:maximum}
On the unit sphere of \(V_\chi\),
\[
        \max_{\|v\|=1} F(v)=\frac1d.
\]
The maximizers are exactly the unit vectors that generate irreducible summands affording \(\chi\).
\end{corollary}

\section{The Universal Algorithm}\label{sec:universal}

Let \(G=\{g_1,\dots,g_N\}\), where \(N=|G|\).  The left regular representation is
\[
        \lambda:G\to GL(\mathbb C[G]),
        \qquad
        \lambda(g)e_h=e_{gh}.
\]
This representation is canonical once an ordering of \(G\) is fixed, and it
contains every irreducible representation of \(G\).  Hence it gives a universal
ambient space for constructing a representation with a prescribed irreducible
character.

\begin{algorithm}[H]
\caption{Character-to-matrices construction}
\label{alg:main}
\begin{algorithmic}[1]
\REQUIRE A finite group \(G\) with computable multiplication, and an irreducible character \(\chi\in\Irr(G)\) of degree \(d\).
\ENSURE Matrices affording an irreducible representation with character \(\chi\).

\STATE Construct the left regular representation \(\lambda:G\to GL(\mathbb C[G])\).
\STATE Form the isotypic projector
\[
        P_\chi
        =
        \frac{d}{|G|}
        \sum_{g\in G}
        \overline{\chi(g)}\,\lambda(g).
\]
\STATE Compute a basis for \(V_\chi=\operatorname{im}(P_\chi)\).
\STATE Choose a vector \(u\in\mathbb C[G]\) with \(P_\chi u\neq 0\), for instance a random or nonsymmetric vector, and set
\[
        v_0=\frac{P_\chi u}{\|P_\chi u\|}.
\]
\STATE Starting from \(v_0\), maximize
\[
        F(v)
        =
        \frac1{|G|}
        \sum_{g\in G}
        |\langle v,\lambda(g)v\rangle|^2
\]
on the unit sphere of \(V_\chi\).
\STATE Let \(v\) be a maximizer, so \(F(v)=1/d\).
\STATE Choose group elements \(h_1,\dots,h_d\in G\) such that
\[
        \lambda(h_1)v,\dots,\lambda(h_d)v
\]
are linearly independent.
\STATE Let \(B\) be the matrix whose columns are these vectors.
\STATE For each generator \(s\) of \(G\), solve
\[
        \lambda(s)B=BA_s
\]
for \(A_s\in M_d(\mathbb C)\).
\RETURN The matrices \(A_s\).
\end{algorithmic}
\end{algorithm}

\begin{theorem}
\label{thm:algorithm}
Algorithm \ref{alg:main} returns matrices defining an irreducible representation of \(G\) with character \(\chi\).
\end{theorem}

\begin{proof}
The regular representation contains the irreducible representation with character \(\chi\), so \(V_\chi\neq 0\).  By Corollary~\ref{cor:maximum}, a unit vector \(v\in V_\chi\) satisfying \(F(v)=1/d\) generates an irreducible subrepresentation
\[
        U=\Span\{\lambda(g)v:g\in G\}
\]
affording \(\chi\).  Since \(\dim U=d\), one can choose \(h_1,\dots,h_d\in G\) such that the vectors \(\lambda(h_i)v\) form a basis of \(U\).  In this basis, the restriction of \(\lambda\) to \(U\) gives matrices of an irreducible representation with character \(\chi\).
\end{proof}

In exact arithmetic, Step 7 only asks for any independent orbit vectors.  In a
floating-point implementation, one should select columns from the orbit matrix
by rank-revealing linear algebra and solve the equations in Step 9 by least
squares.

\section{Gradient Ascent}\label{sec:gradient}

After computing \(V_\chi=P_\chi V\), choose
\[
        v_0=\frac{P_\chi u}{\|P_\chi u\|}
\]
from a generic vector \(u\) with \(P_\chi u\neq 0\).  For
\[
        c_g(v)=\langle v,\rho(g)v\rangle,
\]
the unitary case has real Euclidean gradient, defined by
\[
        dF_v(h)=\operatorname{Re}\langle h,\nabla F(v)\rangle,
\]
or equivalently by \(\operatorname{Re}\langle\nabla F(v),h\rangle\).  It is
\[
        \nabla F(v)
        =
        \frac{4}{|G|}
        \sum_{g\in G}
        \overline{c_g(v)}\,\rho(g)v.
\]
The spherical gradient in \(V_\chi\) is
\[
        \nabla_{\mathrm{sph}}F(v)
        =
        P_\chi\nabla F(v)
        -
        \langle v,P_\chi\nabla F(v)\rangle v,
\]
and one ascent step is
\[
        v\leftarrow
        \frac{v+\eta\nabla_{\mathrm{sph}}F(v)}
        {\|v+\eta\nabla_{\mathrm{sph}}F(v)\|}.
\]
Here \(\eta>0\) is a step size.
In computations we start from a generic projected vector to avoid obvious
symmetries of the initial data.  When \(F(v)=1/d\), the orbit of \(v\) gives
the desired irreducible subrepresentation.

\section{Using Smaller Ambient Representations}

The regular representation always works, but it may be large.  The same
construction applies to any representation \(\rho:G\to GL(V)\), with character
\(\chi_\rho\), satisfying
\[
        \langle \chi,\chi_\rho\rangle_G>0.
\]
Then \(P_\chi V\neq 0\), and the same maximization extracts one irreducible
copy of \(\chi\).

\begin{proposition}[Reduction from an abelian subgroup]
Let \(A\leq G\) be abelian and let \(\chi\in\Irr(G)\).  Then some linear
character \(\theta\in \Irr(A)\) occurs in \(\Res_A^G\chi\).  For any such
\(\theta\), the induced representation \(\Ind_A^G\theta\) contains \(\chi\).
Its dimension is
\[
        \dim \Ind_A^G\theta=[G:A].
\]
\end{proposition}

\begin{proof}
Since \(A\) is abelian, every irreducible character of \(A\) is linear, and
\(\Res_A^G\chi\) is a sum of such characters.  At least one occurs because
\(\chi(1)>0\).  If \(\theta\) occurs, then Frobenius reciprocity gives
\[
        \langle \chi,\Ind_A^G\theta\rangle_G
        =
        \langle \Res_A^G\chi,\theta\rangle_A
        >0.
\]
Thus \(\Ind_A^G\theta\) contains \(\chi\), and its dimension is \([G:A]\)
because \(\theta\) is one-dimensional.
\end{proof}

For a cyclic subgroup \(A=\langle a\rangle\), this choice is completely
explicit: decompose the values of \(\chi(a^j)\) against the linear characters
of \(A\), choose a character with nonzero multiplicity, and induce it to \(G\).

\section{Example: The two-dimensional representation of \(S_4\)}\label{sec:s4-example}

We illustrate the algorithm for the irreducible character \([22]\) of \(S_4\),
using the usual partition notation for irreducible characters of \(S_4\).
Let
\[
        s=(12),
        \qquad
        t=(1234).
\]
We only need the following row of the character table:
\begin{center}
\begin{tabular}{c|ccccc}
\text{cycle type} & \(1^4\) & \(2\,1^2\) & \(2^2\) & \(3\,1\) & \(4\) \\
\hline
\(\chi_{[22]}\) & \(2\) & \(0\) & \(2\) & \(-1\) & \(0\)
\end{tabular}
\end{center}
Thus \(d=2\), and the target value is \(1/d=1/2\).

\subsection*{Regular representation}

Let \(\lambda\) be the regular representation of \(S_4\).  The ambient space has
\[
        \dim \mathbb C[S_4]=24.
\]
Write \(e_g\) for the standard basis vector of \(\mathbb C[S_4]\) indexed by
\(g\), and let \(C_{2^2}\) and \(C_{3,1}\) denote the classes of double
transpositions and 3-cycles, respectively.
The projector onto the \([22]\)-isotypic component is
\[
        P_{[22]}
        =
        \frac1{12}
        \left(
        2\lambda(1)
        +
        2\sum_{g\in C_{2^2}}\lambda(g)
        -
        \sum_{g\in C_{3,1}}\lambda(g)
        \right),
\]
Its image has dimension
\[
        \dim P_{[22]}\mathbb C[S_4]=4.
\]

The canonical vector \(P_{[22]}e_1\) gives the value \(F=1/4\), below the
target value \(1/2\).  Instead we use the nonsymmetric vector
\[
        u=e_1+2e_s+3e_t,
        \qquad
        v_0=\frac{P_{[22]}u}{\|P_{[22]}u\|}.
\]
Applying the formulas in Section~\ref{sec:gradient}, the values progress from
\[
        F(P_{[22]}e_1/\|P_{[22]}e_1\|)=0.250000,
        \qquad
        F(v_0)=0.359375
\]
to \(F(v_\infty)=0.500000=1/2\), where \(v_\infty\) denotes the limiting
vector obtained by the ascent iteration.  Thus the iteration reaches the
maximum, and the orbit of \(v_\infty\) spans one irreducible copy of \([22]\).

Choosing two independent orbit vectors as a basis and restricting \(\lambda(s)\)
and \(\lambda(t)\) gives matrices equivalent to
\[
        \rho(s)
        =
        \begin{pmatrix}
        1&0\\
        -1&-1
        \end{pmatrix},
        \qquad
        \rho(t)
        =
        \begin{pmatrix}
        -1&-1\\
        0&1
        \end{pmatrix}.
\]
The trace check is
\begin{center}
\begin{tabular}{c|cccc}
\text{element} & \((12)\) & \((1234)\) & \((12)(34)\) & \((123)\) \\
\hline
\(\Tr \rho(g)\) & \(0\) & \(0\) & \(2\) & \(-1\) \\
\(\chi_{[22]}(g)\) & \(0\) & \(0\) & \(2\) & \(-1\)
\end{tabular}
\end{center}
Also,
\[
        \rho(s)^2=I,
        \qquad
        \rho(t)^4=I,
        \qquad
        (\rho(s)\rho(t))^3=I.
\]

\subsection*{Reduced ambient representation}

Let
\[
        V_4=\{1,(12)(34),(13)(24),(14)(23)\}.
\]
Take
\[
        \rho=\Ind_{V_4}^{S_4}\mathbf 1_{V_4}.
\]
Here \(\mathbf 1_{V_4}\) denotes the trivial character of \(V_4\).
Then
\[
        \dim \rho=[S_4:V_4]=6,
        \qquad
        \rho\cong [4]\oplus[1111]\oplus 2[22].
\]
So the ambient dimension drops from \(24\) to \(6\), while the \([22]\)-isotypic
component is still 4-dimensional.

Using the corresponding coset-basis vector, where \(\delta_{xV_4}\) denotes
the basis vector indexed by the coset \(xV_4\),
\[
        u=\delta_{V_4}+2\delta_{sV_4}+3\delta_{tV_4},
\]
the same computation, with \(P_{[22]}\) acting on this induced representation,
reaches \(F(v_\infty)=0.500000\).  Thus the algorithm works in dimension \(6\)
instead of dimension \(24\), while the optimization step is still needed.

For comparison, one can choose still smaller ambient representations in which
\([22]\) occurs with multiplicity one.  For instance, the action of \(S_4\) on
the three pair partitions of \(\{1,2,3,4\}\) gives
\[
        \mathbb C\{12|34,13|24,14|23\}
        \cong
        \mathbf 1\oplus [22].
\]
In that model the projection already isolates the desired irreducible
representation.  It is more efficient, but it does not illustrate the
optimization step; this is why the reduced example above uses \(V_4\).

\section{Computational Remarks}

Let \(q=\dim V\) be the dimension of the chosen ambient representation.  The
regular representation corresponds to \(q=|G|\).  If the action matrices are
stored densely, the method is quickly impractical.  In the regular and induced
settings, however, the matrices are usually permutation or monomial matrices,
so one applies \(\rho(g)\) by permuting coordinates and multiplying by scalars.

A full evaluation of \(F\), or one gradient step, requires applying \(\rho(g)\)
for all \(g\in G\).  With sparse, permutation, or monomial actions, this gives
the following rough cost comparison:
\[
\begin{array}{c|c|c}
\text{ambient representation} & \text{dimension } q & \text{cost per full pass} \\
\hline
\mathbb C[G] & |G| & O(|G|^2) \\
\Ind_A^G\theta,\ \theta \text{ linear} & [G:A] & O(|G|[G:A]) \\
\text{general sparse ambient } V & q & O(|G|q)
\end{array}
\]
The same estimates apply to applying the projector \(P_\chi\) to a vector.
Thus replacing the regular representation by an induced representation from a
large abelian subgroup can substantially reduce the computation.

The extraction of an orbit basis should also be done with numerical care.  In
exact arithmetic, it is enough to choose group elements \(h_1,\dots,h_d\) for
which \(\rho(h_i)v\) are independent.  In floating point, one should use a
rank-revealing QR decomposition or SVD on the orbit matrix, then compute the
restricted matrices from \(\rho(s)B=BA_s\), preferably by least squares if the
data are approximate.  Thus the regular representation is a reasonable
baseline for small examples, while induced and permutation representations are
natural practical alternatives for larger groups.

\section{Conclusion}

We have described a direct construction of explicit irreducible matrix
representations from irreducible characters: project to the desired isotypic
component, maximize the norm of a diagonal matrix coefficient, and extract a
basis from the orbit of a maximizing vector.  The regular representation gives
a universal construction, while smaller ambient representations give practical
versions for computation.

\end{document}